\documentclass{article}
\usepackage{amssymb}
\usepackage{amsmath}
\usepackage{amsthm}

\newtheorem{theorem}{Theorem}[section]

\newtheorem{proposition}[theorem]{Proposition}
\newtheorem{remark}[theorem]{Remark}

\newtheorem{definition}[theorem]{Definition}
\numberwithin{equation}{section}

\newcommand{\R}{\mathbb{R}}

\newcommand{\ed}{\end{document}}
\newcommand{\ee}{\end{equation}}

\begin{document}

\title{Duality methods for a class of quasilinear systems}
\author{Antonella Marini$^1$ and
Thomas H. Otway$^2$
\\ \\
\textit{$^1$Dipartimento di Matematica, Universit\`{a} di
L'Aquila,}\\ \textit{67100 L'Aquila, Italy } \\ \textit{$^{1,2}$Department of Mathematical Sciences,
Yeshiva University, }\\ \textit{New York, New York 10033}}
\date{}
\maketitle

\begin{abstract}

 Duality methods are used to generate explicit solutions to nonlinear Hodge systems, demonstrate the well-posedness of boundary value problems,  and  reveal, via the Hodge--B\"acklund transformation, underlying symmetries among superficially different forms of the equations. \textit{MSC2010}: 58A14, 58A15, 35J47, 35J62, 35M10.

\end{abstract}

\section{Introduction}
After well over a half-century, the equations of Hodge and Kodaira remain a fruitful approach to the theory of irrotational fields, which they endow with the rich topological structure of de Rahm cohomology. See, \emph{e.g.}, Ch.\ 7 of \cite{Mr}, or \cite{Sch}, for introductions. A solution to the Hodge--Kodaira equations is a $k$-form $\omega$ which is \emph{closed} ($d\omega=0$) and \emph{co-closed} ($\delta\omega=0$) under the exterior derivative $d,$ where $\delta$ is its formal adjoint.

Most of the interesting classical fields are quasilinear. The nonlinear Hodge theory conjectured by Bers and realized by Sibner and Sibner \cite{SS1} introduces Hodge-like equations which model irrotational velocity fields associated with steady, ideal compressible flow. In that extension, the requirement of classical Hodge theory that the solution $\omega$ be co-closed under exterior differentiation is weakened to the requirement that only the product of $\omega$ and a possibly nonlinear term $\rho$ must have this property.

Classical fields are frequently characterized by vortices. So although most conservative field theories are quasilinear, most quasilinear field theories are not conservative (even locally), and it is worthwhile to study the analytic properties of equations in which the requirement that the solution be closed under exterior differentiation is also weakened. Thus in a recent paper \cite{MO1} we studied the invariantly defined system (\cite{O1}, Sec. VI; \cite{OPad}, Sec.\ 4)
\begin{equation}\label{HF1-2}
\left\{\begin{array}{ll}
    \delta\left(\rho(Q) \omega\right)=0\\
    d\omega = \Gamma\wedge \omega
    \end{array}\right.
\end{equation}
for unknown $\omega\in \Lambda^k(\Omega),$ $k\in\mathbb{Z}^+,$ with $\Omega$ a smooth open domain in $\R^n$,
and continuously differentiable $\Gamma\in \Lambda^1(\Omega).$ Here $Q=\vert \omega\vert^2=\ast\left(\omega\wedge\ast\omega\right),$ with $\ast$ denoting the Hodge duality operator $\ast:\Lambda^k(\Omega)\rightarrow\Lambda^{n-k}(\Omega);$  $\rho$ is a positive, H\"older-continuously differentiable
function of $Q,$ which is generally given by the physical or geometric context. We call (\ref{HF1-2})
the \emph{nonlinear Hodge--Frobenius equations}, as they generalize the nonlinear
Hodge equations
\begin{equation}\label{NLHe}
\left\{\begin{array}{ll} \delta\left(\rho(Q) \omega\right)=0\\
d\omega=0
\end{array}\right.
\end{equation}
introduced in \cite{SS1}.
In this paper we study (\ref{HF1-2}) and also variants in which the term $\Gamma$ in the second equation --
the \emph{Frobenius condition} --
may depend on $\omega,$ or in which the co-differential equation assumes a special inhomogeneous form and $\rho=\rho(\mathbf x, Q)$ may depend explicitly on $\mathbf x\in\Omega$.

The Frobenius condition represents a weakening of the local conservation hypothesis $d\omega=0$ in system  (\ref{NLHe}). The resulting field is no longer locally conservative, but generates a closed ideal. For this reason, it is completely integrable (in the sense of Frobenius) for forms of degree or co-degree equal to 1, or for general $k$ under the additional hypothesis that $\Gamma$ be exact; see, \emph{e.g.}, \cite{E}, Sec.\ 4-2. The hypothesis that $\Gamma$ be exact is automatically satisfied in the case $k=1$ or $k=n-1.$  If $\Gamma$ is exact,  say  $\Gamma = d\eta$ for $\eta\in\Lambda^0(\Omega),$ solutions to eqs.\ (\ref{HF1-2}) are locally exact when multiplied by an integrating factor; that is, they  have the local  form
$$e^\eta\omega = d\Psi$$
for $\Psi\in \Lambda^{k-1}(\Omega);$ see the discussions in Secs.\ 2.1 and 2.2 of \cite{MO1} and in Sec.\ 1 of \cite{MO2}.
For many applications the weaker condition
\begin{equation} \label{Hcon}
\omega\wedge d\omega=0
\end{equation}
suffices in place of the Frobenius condition; see, \emph{e.g.}, Sec.\ 1.2 of \cite{MO1}. In cases for which the Frobenius condition is used only to imply (\ref{Hcon}), or for cases in which it is interpreted as a condition for an integrating factor, the 1-form $\Gamma$ need not be prescribed: any nonsingular $\Gamma$ will do.

Diverse choices of the mass density $\rho$ arise in models of classical fields. These models are reviewed in \cite{O4}, Sec. 2.7 and Chs.\ 5 and 6. Most classical fields which satisfy quasilinear partial differential equations are vectorial, and these vectorial solutions correspond via isomorphism to 1-form solutions of  \eqref{HF1-2} or \eqref{NLHe}. But occasionally there are matrix-valued solutions of quasilinear field equations, and some of these correspond to 2-form-valued solutions of the nonlinear Hodge or Hodge--Frobenius equations. Examples of equations having 1-form solutions include the continuity equations for the velocity field of a steady, compressible fluid flow \cite{SS1} and for certain models of shallow hydrodynamic flow \cite{O4}. Examples of equations having 2-form solutions include nonlinear Maxwell's equations for electromagnetic fields \cite{MP}, Born--Infeld  fields \cite{Y}, and certain twisted variants of these \cite{O1}, \cite{SSY}. The variety of applications discussed in \cite{O4} and the references cited therein suggest that eqs. \eqref{HF1-2} and \eqref{NLHe} are rather generic: they apply, under various additional hypotheses, to a wide variety of models. For this reason, it is worthwhile to study their analytic properties, as we do here and in \cite{MO1}, without focusing on any particular application.

\subsection{Organization of the paper}
In Sec. \ref{S-Aharmonic}  we derive the existence of solutions to a Hodge--Frobenius system, in which the solution is co-closed and the Frobenius condition is nonlinear, from the existence of an appropriate class of $A$-harmonic forms.

In Sec. \ref{S-construct} we give an algebraic criterion for inverting the operator $A$. That criterion can be applied also to the hyperbolic range of the corresponding  nonlinear Hodge-Frobenius system. We use this inversion to write an explicit formula for the solutions to the system and  generate a concrete example.

In Sec.  \ref{S-HB} we show that certain superficially different models for classical fields can be shown to be Hodge--B\"acklund transforms of each other. In that section we transform different types of nonlinear Hodge--Frobenius systems, including a variational form of these systems,
 into nonlinear Hodge systems (\ref{NLHe}) of particular type.

 In Sec.  \ref{S-BVP} we prove the existence and uniqueness of solutions  to boundary value problems of Dirichlet and Neumann type  in the elliptic regime, for inhomogeneous nonlinear Hodge--Frobenius systems in which the $1$-form $\Gamma$ in the Frobenius condition is exact. We do so for both linear and nonlinear Frobenius conditions. The results of Sec. 5 are  an application of the results obtained in Sec. \ref{S-HB} and of known results for the conventional nonlinear Hodge equations (\ref{NLHe}) in the elliptic range.

\section{Relation to $A$-harmonic forms}\label{S-Aharmonic}

It was observed in Sec.\ 1 that the Frobenius condition emerges as a natural weakening of the conservation hypothesis $d\omega = 0$.
But the Hodge--Frobenius equations also arise naturally from the nonlinear Hodge equations (\ref{NLHe}) in a completely different way, as a dual, or conjugate form of the equations. The use of conjugate forms in nonlinear Hodge theory goes back at least to \cite{SS1}.
Dirichlet and Neuman problems for eqs.\ (\ref{NLHe}) were introduced in \cite{SS2}.

If $u\in\Lambda^{k-1}$ and $v\in\Lambda^{k+1},$ then the Cauchy--Riemann
equations can be written in the form $du=\delta v.$ More generally,
we may consider $A$-harmonic extensions.
We call the differential forms $u\in\Lambda^{k-1}$ and $v\in\Lambda^{k+1}$ \emph{ conjugate A-harmonic forms} if they satisfy the equation
\begin{equation}\label{CRgen}
A\left(x,du\right)=\delta v\,,
\end{equation}
where $A\,:\, \Omega\times \Lambda^{k}(\Omega)\to\Lambda^k(\Omega)$ is a differential operator of order $0$ and $\Omega$ is a domain of $\mathbb{R}^n$; see,
\emph{e.g.}, \cite{AD} for an exposition and \cite{ISS} for analytic properties.

 We specify $A$ to be given by
\begin{equation}\label{Adef}
A\left(x,\omega\right) = A(\omega)=\rho(|\omega|^2)\omega,\quad \omega\in \Lambda^k(\Omega)\,,
\end{equation}
and impose further conditions on $A$ or $\Omega$ as we require them.
Our immediate goal is to define Hodge--Frobenius fields in terms of conjugate $A$-harmonic $k$-forms.
We say that A is \emph{invertible} if
there exists an operator $B\,:\, \Omega\times \Lambda^k(\Omega)\to\Lambda^{k}(\Omega)$ such that
$$\begin{array}{ll}
B\left(x, A(x,\omega)\right) = \tilde\omega \,,\notag\\
A\left(x, B(x,\tilde\omega)\right)=\omega \quad\forall \omega,\, \tilde\omega\in\Lambda^{k}(\Omega)\,.\notag\\
\end{array}$$
Associated to $A$  is the differential operator $\tilde A$ of order 1,
$
\tilde A\,:\, \Omega\times \Lambda^{k-1}(\Omega)\to\Lambda^k(\Omega),$
$(x, u)\to A(x, du)$,
and the second order differential equation
$\delta \tilde A(x, u) = 0$ (with its inhomogeneous variants),
of which the co-differential equation in \eqref{HF1-2}, in the special case of $\omega$ exact  and $\rho(Q) = Q^{p/2},$ is the $p$-harmonic equation.
\begin{proposition} \label{Aharm} Let $A$ be given by (\ref{Adef}), with $\rho$ sufficiently smooth and positive.
 Assume $A$ to be invertible. Let $u\in\Lambda^{k-1}(\Omega)$ and $v\in\Lambda^{k+1}(\Omega)$ be sufficiently smooth, conjugate $A$-harmonic forms.
 Then $\tilde\omega\equiv \delta v=A(du)\in\Lambda^k(\Omega)$ is a solution to the Hodge--Frobenius equations in the form
\begin{equation}\label{linHF1-2}
\left\{\begin{array}{ll}
\delta\tilde\omega=0\\
d\tilde\omega = \Gamma \wedge\tilde\omega\\
\end{array}\right.
\end{equation}
with $\Gamma  \equiv   d \ln\rho(\vert B\left(\tilde{\omega}\right)\vert^2),$ where $B\equiv A^{-1}.$
Conversely, for any given $\tilde\omega\in\Lambda^k(\Omega)$ satisfying (\ref{linHF1-2}), with
$\Gamma  \equiv   d \ln\rho(\vert B\left(\tilde{\omega}\right)\vert^2)$,
the $k$-form $\omega  \equiv  B(\tilde\omega)$ satisfies eqs. \eqref{NLHe}.
 If in addition $\Omega$ is contractible, then $\tilde\omega =\delta v$,  $\omega=du$ for some conjugate A-harmonic forms  $u\in \Lambda^{k-1}(\Omega)$, $v\in\Lambda^{k+1}(\Omega).$ \end{proposition}
\begin{proof} To prove the first assertion we proceed as follows. The co-closedness of $\tilde\omega$ comes directly from the fact that the generalized  Cauchy--Riemann equations \eqref{CRgen} are satisfied and that $\delta^2=0$ on differential forms of class $\mathcal C^2$.
Furthermore,
\begin{equation}\label{Beq}
du = \frac{1}{\rho (\vert du\vert^2)} A(du) = \eta(\vert \delta v\vert^2) \delta v\,,
\end{equation}
with $\eta(\vert \delta v\vert^2)$  (well) defined by the formula
$\eta(\vert \delta v\vert^2)\rho(\vert B\left(\delta v\right)\vert^2) = 1.$ We conclude that $\eta (\vert\delta v\vert^2)>0$, as $\rho(\vert du\vert^2)>0$ by hypothesis. Having set $\tilde{\omega} \equiv\delta v,$
(\ref{Beq}) implies  $0=d^2 u = d \,\left(\eta(\vert\tilde\omega\vert^2)\tilde\omega\right).$ This
yields the  nonlinear Frobenius condition in \eqref{linHF1-2}
with
$\tilde{\eta} (\vert\tilde\omega\vert^2)\equiv -  \ln\eta(\vert\tilde\omega\vert^2) =  \ln\rho(\vert B\left(\tilde\omega\right)\vert^2)\,.$

 Conversely, substituting $A(\omega)= \rho(\vert\omega\vert^2)\omega$ for $\tilde \omega$ in the first equation  in
  \eqref{linHF1-2}, one obtains the first equation in \eqref{NLHe}.
 Likewise,  the second equation in  \eqref{linHF1-2} can be multiplied by $e^{-\eta}$ and rewritten as
  $$0= d\left(\tilde\omega e^{-\tilde\eta(\vert\tilde\omega\vert^2)}\right)=d \left(A(\omega)\,
  \frac{1}{\rho(\vert \omega\vert^2)}\right) = d\omega\,.$$
 If $\Omega$ is a contractible  domain,  the application of the Poincar\`e Lemma and its adjoint version to $\omega$ and $\tilde\omega$, respectively,  completes the proof.\end{proof}

The following proposition
gives a partial converse to Prop. \ref{Aharm}. A systematic  approach to  the study of the invertibility of the operator $A$, leading to a method to construct explicit solutions to eqs. \eqref{HF1-2}, is postponed until Sec. \ref{S-construct}.

\begin{proposition}\label{P-converse}
Let $\tilde\eta\,:\, \R^+\cup \{0\}\to \R^+$ be a prescribed smooth function and $I$ be an interval such that the function  $f\,:\, \tilde t\to t\equiv \tilde t \exp \left[-2\tilde\eta(\tilde t)\right]$ restricted to $I$ is  1:1. Let $\rho\,:\, f(I)\to \R^+$ be defined by $\rho(t)=\exp\left[\tilde\eta\left(f^{-1}(t)\right)\right]$. Then for each  $\tilde\omega\in\Lambda^k(\tilde \Omega)$ satisfying (\ref{linHF1-2}) with $\Gamma =d\left[\tilde\eta (\vert\tilde\omega\vert^2)\right]$ there exists a unique $\omega\in \Lambda^k(\Omega)$ satisfying $A(\omega) = \tilde\omega$, with $A$ defined as in \eqref{Adef}.
Such $\omega$ also satisfies \eqref{NLHe} with $\rho$ as prescribed.
Conversely, if $\omega\in\Lambda^k(\Omega)$ satisfies system (\ref{NLHe}) with $\rho$ as prescribed, then the differential form $\tilde\omega\equiv A(\omega)$ satisfies \eqref{linHF1-2} with $\Gamma=  d \tilde\eta$,   $\tilde \eta$ as prescribed.   If the domains $\Omega$ and $\tilde\Omega$ are contractible, our assertions are true with $\omega$ replaced by  an exact form $du$  and $\tilde\omega$ replaced by a co-exact form $\delta v$, yielding  conjugate $A$-harmonic forms $u, v$.
Moreover,  $\omega$ satisfies homogeneous Dirichlet or Neumann conditions on $\Omega$ if and only if $\tilde{\omega}$ does as well.
\end{proposition}
\begin{proof} Let $\tilde\omega$ satisfy  (\ref{linHF1-2}) with $\Gamma =d\left[\tilde\eta (\vert\tilde\omega\vert^2)\right]$. Then the differential form
$\omega\equiv \exp\left[ -\tilde\eta(\vert\tilde\omega\vert^2)\right]\tilde\omega$ satisfies
$$A(\omega) \equiv \rho(\vert\omega\vert^2)\omega\equiv e^{\tilde\eta\left(f^{-1}(\vert\omega\vert^2)\right)}\omega= \tilde\omega\,.$$
For $\rho$ as prescribed, suppose that the differential $k$-forms  $\omega_1$, $\omega_2$ satisfy
$$\rho(\vert\omega_1\vert^2) \,\omega_1=\tilde\omega =\rho(\vert\omega_2\vert^2) \,\omega_2\, ,\qquad \vert\omega_j\vert^2\in f(I)\,.$$
From this we see that $\omega_1=\omega_2$ if and only if $\vert\omega_1\vert=\vert\omega_2\vert.$  By taking absolute values and squaring that formula, we also conclude that  $\vert\omega_1\vert=\vert\omega_2\vert$ is the unique inverse image under $f$ of $\vert\tilde\omega\vert^2.$ Thus $\omega_1=\omega_2.$

As $\tilde\omega = \rho(\vert\omega\vert^2) \omega$,  homogeneous Dirichlet  or Neumann boundary conditions for $\omega$ become homogenous Dirichlet or Neumann boundary conditions for $\tilde\omega$.

The remainder of the proof is contained in the proof of Proposition \ref{Aharm}.

\end{proof}
\begin{remark}
 Proposition \ref{P-converse} gives a precise correspondence between solutions of the Hodge--Frobenius equations \eqref{linHF1-2}  with non-linear constraint and solutions of  \eqref{NLHe}. Such a correspondence provides  the basis to obtain  existence and uniqueness theorems for Dirichlet or Neumann problems
from analogous theorems for the conventional  nonlinear Hodge theory; see \cite{SS2}.
In Section  4 this  correspondence is extended to systems of the form  \eqref{HF1-2} under
conditions on $\Gamma$ and on the density function  in \eqref{HF1-2} sufficient  to  guarantee the ellipticity condition
\begin{equation} \label{elliptic}
 0<\rho^2(Q)+2Q\rho ^{\prime}(Q)\rho(Q)
\end{equation}
for the transformed system. It is necessary to assume  appropriate smoothness of the boundary of the domain, of the coefficients of the equation  and of $\omega$ in order to guarantee the well-posedness of the Dirichlet and Neumann problems; \emph{cf.} Sec. \ref{S-BVP}, Theorems \ref{DE!}, \ref{NE!}, and  Theorems 1 and 2 of \cite{SS2}.   \end{remark}

\section{The construction of solutions}\label{S-construct}

We now want to use the operator $A$ defined in \eqref{Adef}
to prove results which are independent of equation type.
In Proposition \ref{Aharm}  we assumed the existence of an inverse for the quasi-linear coefficient $A$.
In this section we define conditions under which that hypothesis is satisfied.
\begin{theorem}\label{Tinvert}
Let $A$ be defined via the formula \eqref{Adef}.
Assume that $\rho$ is such that the function
\begin{equation}
\label{phi}
\phi_\rho(t)\equiv t\rho^2(t)\,,
\end{equation}
when restricted to the connected interval $(t_1, t_2)$, satisfies
\begin{equation} \label{monoton}
\frac{d\phi_\rho}{dt}>0 \,\;\text{ or } \,\, \;\frac{d\phi_\rho}{dt}<0\,.
\end{equation}
Let $\Lambda^k(\Omega)_{t_1,t_2}$ denote the set of differential $k$-forms $\omega$ such that
$t_1\leq \vert \omega\vert^2 \leq t_2,$ and let $(r_1, r_2)$ be the image under $\phi_\rho$ of the interval $(t_1, t_2).$ Then
\begin{equation}\label{restrict}
A_{\left|\Lambda^k(\Omega)_{t_1,t_2}\right.}\,:\,  \Lambda^k(\Omega)_{t_1,t_2} \to \Lambda^k(\Omega)_{r_1,r_2}\,,
\end{equation}
and its restriction to $\Lambda^k(\Omega)_{t_1,t_2}$ is invertible with inverse
\begin{equation} \label{B1-2}
\begin{array} {lll}
B\,:\, \Lambda^k(\Omega)_{r_1,r_2}\to \Lambda^k(\Omega)_{t_1,t_2}\,,
\qquad \tilde\omega\to \tilde\omega   /  \rho(\psi(\vert \tilde \omega\vert^2))\,,\\
\hfill\\
\qquad \mbox{with }\; \psi \equiv  {{\phi_\rho}_{\left|(t_1, t_2)\right.}}^{-1}\,:\, (r_1, r_2)\to (t_1, t_2)\,. \\
\end{array}
\end{equation}
\end{theorem}
\begin{proof}
Condition \eqref{monoton} implies by monotonicity that there exists an inverse $\psi\,:\, (r_1, r_2)\to (t_1, t_2)$ of the map $\phi_\rho$ defined in \eqref{phi} on the interval $\left(t_1,t_2\right).$ Condition \eqref{restrict} is satisfied because
\[
\vert A(\omega)\vert^2 \equiv \vert\rho(\vert \omega\vert^2) \omega\vert^2 = \rho^2(\vert\omega\vert^2)\vert\omega\vert^2\equiv \phi_\rho(\vert\omega\vert^2)\,,
\]
with $t_1 \leq \vert\omega\vert^2\leq t_2.$ Similarly, for $k$-forms  $\tilde{\omega}\in \Lambda^k(\Omega)_{\tau_1,\tau_2}$ and $B$ defined by the formula in \eqref{B1-2}  one has
\begin{equation} \label{star}
\vert B\left(\tilde{\omega}\right)\vert^2=\frac{\vert\tilde{\omega}\vert^2}{\rho^2\left(\psi\left(\vert\tilde{\omega}\vert^2\right)
\right)}=\frac{\vert\tilde{\omega}\vert^2\psi\left(\vert\tilde{\omega}\vert^2\right)}{\rho^2\left(\psi\left(\vert\tilde{\omega}\vert^2\right)
\right)\psi\left(\vert\tilde{\omega}\vert^2\right)}\,.
\end{equation}
The denominator in \eqref{star} can be rewritten as
\begin{equation}\label{invpsi}
\psi\left(\vert\tilde{\omega}\vert^2\right)\rho^2\left(\psi\left(\vert\tilde{\omega}\vert^2\right)\right) = \phi_\rho (\psi\left(\vert\tilde{\omega}\vert^2\right)  =\vert\tilde{\omega}\vert^2\,,
\end{equation}
yielding  $\vert B\left(\tilde{\omega}\right)\vert^2=\psi\left(\vert\tilde{\omega}\vert^2\right)\in\left(t_1,t_2\right).$ That is, $B\,:\,\Lambda^k(\Omega)_{r_1,r_2}\to \Lambda^k(\Omega)_{t_1,t_2}$.
For $k$-forms $\omega\in\Lambda^k(\Omega)_{t_1,t_2}$ we have
\[
B\left( A(\omega)\right) = \frac{A(\omega)}{ \rho(\psi(\vert A(\omega)\vert^2))}=
\frac{\rho(\vert \omega\vert^2) \,\omega}{\rho(\psi(\rho^2(\vert\omega\vert^2)\vert \, \omega\vert^2))}=  \frac{\rho(\vert \omega\vert^2)\, \omega}{\rho(\psi(\phi_\rho(\vert\omega\vert^2)))} =\omega\,.
\]
Likewise, for $k$-forms  $\tilde \omega\in \Lambda^k(\Omega)_{r_1,r_2}$  we have
\[
A\left(B(\tilde\omega)\right)= \rho(\vert B(\tilde\omega)\vert^2) B(\tilde\omega)= \rho\left(\frac{\vert\tilde\omega\vert^2}{\rho^2(\psi(\vert\tilde\omega\vert^2))}\right)\frac{\tilde\omega}{\rho(\psi(\vert\tilde\omega\vert^2))}=\tilde\omega\,,
\]
in which, for the last equality, we have divided \eqref{invpsi} by $\rho^2(\psi(\vert\tilde\omega\vert^2))$ and substituted the result into this equation.
This concludes the proof.\end{proof}
 Note that  the conditions in \eqref{monoton} are precisely the conditions that make the system (\ref{NLHe}),  and also (\ref{HF1-2}) with linear Frobenius condition,  either elliptic   (if $d\phi_\rho/dt>0$) or hyperbolic  (if $d\phi_\rho/dt<0$); \emph{cf.}
 (\ref{elliptic}).
\begin{remark}
 We have divided by $\rho$ at various steps of the proof of Theorem \ref{Tinvert}. Clearly this can be done if  $\rho = \rho(t)>0$ $\forall t\in\R^+\cup \{0\}.$ Nonetheless,
the milder assumption $\rho\left(\psi (\vert\tilde\omega\vert^2)\right)\neq 0$ is sufficient for the purpose of finding  smooth solutions to the equation $A(\omega)= \tilde \omega$ with  prescribed $\tilde\omega.$  In some applications, this assumption can be weakened furthermore; cf. \cite{MO2}. \end{remark}

Theorem \ref{Tinvert}  can be used to construct explicit $k$-form-valued solutions to the nonlinear co-differential equation $\delta\left(\rho(\vert\omega\vert^2)\omega\right)=0$ in \eqref{HF1-2}.
For a detailed exposition  of the method and the construction of various examples, see  \cite{MO2}.
Briefly, one argues by the Poincar\'e Lemma that a solution $\omega$ on a contractible domain of $\mathbb R^n$ always admits a ``stream $\left(n-k-1\right)$-form" $f$, that is, a form $f$ satisfying
$\rho(Q)\,\omega=\ast df\,.
$
Theorem \ref{Tinvert} can then be applied directly to obtain the solution formula
\begin{equation} \label{vareqsol}
\omega = \frac{\ast d f}{\rho\left(\psi\left(\vert d  f\vert^2\right)\right)}\,,
\end{equation}
where $\psi$ denotes the inverse(s) of the function $\phi_\rho$ given  by  \eqref{phi}.
The classical solutions $\omega$ are well defined
except possibly at the \emph{sonic hypersurface} dividing the elliptic from the hyperbolic regime. The singular set  will depend on $f$,  $\rho$ and $\psi.$
Sometimes  it is possible to define $\omega$ with continuity, or even higher regularity,  across the sonic hypersurface; \emph{cf.} \cite{MO2}, Sec. 5.1.1. In general such a property is not achieved  (see Ch.\ 6 of \cite{O4} and references cited therein).
On non-contractible domains, one can still write \eqref{vareqsol}  and produce examples of solutions to the co-differential  equation in \eqref{HF1-2}.  More  generally, one can replace the exact forms $df$ in \eqref{vareqsol} by  prescribed closed $(n-k)$-forms.
Satisfaction of the Frobenius condition
for some  $\Gamma$ can be shown and is equivalent to the existence of an integrating factor in the cases $k =1,  n-1$; \emph {cf.}
\cite{MO2}.

As an example, let us consider system \eqref{HF1-2} with prescribed density
$\rho(Q) = \vert Q-1\vert^{-1/2},\; Q\neq 1\,,$
for a differential form of degree $2$ in $4$ dimensions.
This choice of $\rho$ corresponds to the Euclidean Born--Infeld model if $Q<1$ and to the Lorentzian Born--Infeld model if $Q>1$.
All non-cavitating classical solutions $\omega$ can be expressed  by \eqref{vareqsol} on contractible domains $\Omega.$ Cavitating solutions may be expressed by \eqref{vareqsol} as a limit.
In this example, the function $\phi_\rho$ appearing in \eqref{phi} is
$$\phi_\rho(Q) = \frac{Q}{|Q-1|}\,,\; Q\neq 1\,,$$
with inverses $\psi_+\equiv{[{\phi_\rho}_{|_{[0,1)}}]}^{-1}\,:\, [0,\infty)\to[0,1),$ $\psi_-\equiv{\phi_\rho}_{|_{(1,\infty)}}^{-1}\,:\, (1,\infty)\to (1,\infty),$ given by
$$
\psi_ {\pm}\,:\,\xi\rightarrow\frac{\xi}{\xi \pm 1}\,.$$
Corresponding to these inverses of $\phi_\rho$, one obtains the families of solutions
\begin{equation}
\label{sol-pm}
\mathcal W_{\pm} = \left\{ \mathbf{\omega_ {\pm}} = \frac {* d f }{\sqrt{\vert d f\vert^2 \pm1}}\,\; \mbox { with } f\in \Lambda^1(\Omega)\,\right\}\,,
\end{equation}
with the solutions in $\mathcal W_+$ being defined (and uniformly bounded) for smoothly prescribed generalized stream 1-forms $f$, and the solutions in $\mathcal W_-$ requiring the additional condition  $\vert d f\vert>1.$ The family $\mathcal W_-$ contains unbounded solutions corresponding to choices of generalized stream forms  which satisfy $\vert d f\vert=1$ at points of the domain $\Omega.$
One may also prescribe generalized stream forms $f$ such that $\vert d f\vert\to \infty$ when approaching a smooth manifold, say  $\gamma_\infty,$ contained in $\Omega.$  As $\vert\omega_\pm\vert\to 1$ when approaching $\gamma_\infty,$ one may in some cases patch together  the two types of solutions with some regularity.  But the co-differential equation in   \eqref{HF1-2}  would not be satisfied on $\gamma_\infty$ (as $\rho$ would blow up).
Differential forms  $\omega_+\in\mathcal W^+$  and $\omega_-\in \mathcal W^-$ that satisfy a  linear Frobenius condition would then solve  (\ref{HF1-2}) in the elliptic  and hyperbolic regime, respectively.

\section{Hodge-B\"acklund transformations of solutions} \label{S-HB}

One finds in the literature a bewildering redundancy of choices for the mass density $\rho;$ see Sec.\ 1 of \cite{MT}, Sec.\ 2 of \cite {MT2}, and the pairs of densities discussed in \cite{O4}, Sec.\ 2.7 and Ch.\ 6,   in connection with the Born--Infeld and extremal surface equations. It is natural to wonder whether there is a mathematical operation underlying the varieties of density. In this section we extend Theorem 6.1 of \cite{MO1}, which related two particular densities by an application of the Hodge--B\"acklund transformation; see also the special cases studied in \cite{AA}, \cite{AP}, \cite{Le}, \cite{SSY},    \cite{Y}.
A different motivation for seeking a relation between pairs of densities comes from the fact that when we introduce Hodge--B\"acklund transformations we acquire a inhomogeneous right-hand side which has a natural variational  interpretation.  In fact,
the Euler--Lagrange equation for the nonlinear Hodge energy
\begin{equation} \label{NHEn}
    E_{NH} = \frac{1}{2}\int_M\int_0^Q\rho(s)ds\,dM\,,
\end{equation}
where $M$ is an $n$-dimensional Riemannian manifold
and $\Gamma$ is prescribed,
is the inhomogeneous equation (\emph{cf.} \cite{MO1},  Sec.\ 5.1)
    $$\delta\left[\rho(Q)\omega\right]=(-1)^{n(k+1)}\ast \left( \Gamma \wedge \ast \rho(Q)\omega\right)\,.$$
    \begin{definition} \label{dualmassg1} We define the  pair of continuously differentiable densities $\left(\rho, \hat\rho\right)$ to be a \emph{dual pair} if $\rho\,:\,I\subset\R^+\cup \{0\}\to \R^+$,  $\hat\rho\,:\, \hat I  \equiv  \phi_\rho(I)\to \R^+$, with $\phi_\rho \,:\,t\in I\,\to\,\hat t \equiv  t\rho^2(t)$, and the pair $\left(\rho, \hat\rho\right)$ satisfies the identity
\begin{equation}\label{dualmassg}
    \rho(t)\hat\rho(\hat t\,)\equiv 1 \,.
\end{equation}\end{definition}
 Definition \ref{dualmassg1} implies that  the functions $\phi_\rho$ and $\hat\phi_{\hat\rho}$, defined analogously, are inverses of one another; thus both are 1:1.
 In fact, by squaring and multiplying by $t$ throughout, one obtains $t= t\rho^2(t)\hat\rho^2(\hat t\,) = \hat t \hat\rho^2(\hat t\,)= \hat\phi_{\hat\rho}(\hat t\,)$.
 Therefore, the relation of duality defined above is symmetric. For the same reason,  ellipticity of the system \eqref{HF1-2} or \eqref{NLHe} is preserved under the transformation $\rho \to\hat\rho$.   Moreover, the relation \eqref{dualmassg} defines $\hat \rho$ in terms of $\rho$ and vice versa.

\smallskip
An example of a dual pair of densities is the pair $(\rho, \hat\rho)$, with $\rho(t)= 1/ \sqrt{1+t}$ -- associated in the applications with the Born--Infeld model and with the  minimal surface equation -- and   $\hat\rho (t) = 1/ \sqrt{1-t}$ with $ t <1$, associated with the maximal surface equation. The density $\rho(t) = 1/ \sqrt{t-1}$ with $t>1$ is self-dual and is associated with extremal surfaces in Minkowski space.

We find  in the following proposition that systems having  the form \eqref{HFg} can be related to each other by Hodge--B\"acklund transformations.
\begin{proposition}
\label{HBg}
Let $\Sigma,$  $\Gamma$ be given, continuous differential 1-forms, $\left(\rho, \hat\rho\right)$ be a prescribed  dual pair  of  densities.
Then the  k-form  $\omega$ satisfies the nonlinear Hodge--Frobenius system
\begin{equation}\label{HFg}
\left\{\begin{array}{ll}
d* \left(\rho(\vert\omega\vert^2)\omega\right) =\Sigma \wedge * \left(\rho(Q)\omega\right)\\
d\omega = \Gamma\wedge \omega
\end{array}\right.
\end{equation}
if and only if the ($n-k$)-form $\xi\equiv *\left(\rho(\vert\omega\vert^2)\omega\right)$  satisfies the dual  system
\begin{equation}\label{HF*g}
\left\{\begin{array}{ll}
d* \left(\hat\rho(\vert\xi\vert^2)\xi\right) =\Gamma \wedge * \left(\hat\rho(\vert\xi\vert^2)\xi\right)\\
d\xi = \Sigma\wedge \xi
\end{array}\right.
\end{equation}
\end{proposition}
\begin{proof} Multiplying the definition $\xi\equiv *\left(\rho(\vert\omega\vert^2)\omega\right)$ by $\hat\rho(\vert\xi\vert^2)$ and using \eqref{dualmassg}, we obtain
 $\ast \hat\rho(\vert\xi\vert^2)\xi  =  \ast_{n-k}\ast_k \omega  \equiv   \sigma_k\,\omega,$  where the value of $\sigma_k=\pm1$  depends on the order $k$ of the differential form $\omega$ and on the dimension $n$ of the domain $\Omega$.
By the second equation in \eqref{HFg} this yields
\[
d(*\hat\rho(\vert\xi\vert^2)\xi)= d \left(\sigma_k\, \omega\right) = \sigma_k \,d\omega =
(\sigma_k)^2 \,\Gamma\wedge *(\hat \rho (\vert\xi\vert^2)\xi) = \Gamma \wedge * (\hat \rho (\vert\xi\vert^2)\xi)\,,
\]
which is the first equation in the system \eqref{HF*g}. The second equation in \eqref{HF*g} is the first equation in the system \eqref{HFg} with a change in notation.
\end{proof}

If $\Gamma=\Sigma\equiv 0,$ Proposition \ref{HBg}  yields the standard Hodge duality result
for the conventional nonlinear Hodge equations (\ref{NLHe}); see \cite{ISS}, \cite{SS1}.

\begin{theorem} \label{use1} Let $\eta,\zeta\,:\, I\subset \R^+\cup\{0\}\to \R^+\cup\{0\}$ be prescribed continuously differentiable functions, with the additional hypothesis on $\eta$ that the function $f_\eta\,:\,t\in I\to t \exp [- 2\eta(t)]\in \R^+\cup\{0\}$ be invertible with inverse $g_\eta$.
 Let the terms $\Sigma$, $\Gamma$ and the mass density $\rho$
be prescribed by $\Sigma = d\left[\zeta (\vert \omega\vert^2)\right]$,  $\Gamma=d\left[\eta\left(\vert \omega\vert^2\right)\right]$, $\rho=\rho_1(x, \vert\omega\vert^2),$
for $\vert \omega\vert^2\in I$  in  \eqref{HFg}.
Then for every classical solution $\omega_1$ of system \eqref{HFg}  there is a classical solution
$\omega_0$ of the conventional nonlinear Hodge equations \eqref{NLHe} with mass density $\rho_0(x, \vert\omega_0\vert^2),$ where $\rho_0$ depends on
$\rho_1$, $\eta$ and $\zeta$ and $\omega_0$ is related to $\omega_1$ by $C^1$ conformal transformations. The ellipticity condition for system \eqref{NLHe} holds if and only if $g_\eta^\prime$ and $\partial_t \phi_{\rho_1 e^{-\zeta}} $ have the same sign.
The converse also holds.
\end{theorem}
\begin{proof} Let $\omega_1$ be a differential form satisfying $\vert\omega_1\vert^2\in I$ and $\rho =\rho_1\left(x, \vert\omega_1\vert^2\right)$ be a prescribed density function.
Define
\begin{equation}\label{w0}
\omega_0= e^{-\eta{(\vert\omega_1\vert^2)}}\omega_1\,.
\end{equation}
Then   $\vert\omega_0\vert^2\in f( I)$ and $\vert\omega_1\vert^2= g_\eta (\vert\omega_0\vert^2)$. This enables us to define a density function
\begin{equation}\label{rho0}
\rho_0 (x, \vert\omega_0\vert^2) =e^{\eta\left(g_\eta(\vert\omega_0\vert^2)\right)-\zeta\left( g_\eta(\vert\omega_0\vert^2)\right)}\rho_1\left(x, g_\eta(\vert\omega_0\vert^2)\right)\,.\end{equation}
Conversely, given a differential form $\omega_0$ satisfying $\vert\omega_0\vert^2\in f(I)$, a prescribed density function $\rho =\rho_0\left(x, \vert\omega_0\vert^2\right)$,   and functions $\eta$ and $\zeta$ as defined in the hypotheses of this theorem, one can rewrite definition  \eqref{w0} as
$$
\omega_1= e^{\eta{\left(g_\eta(\vert\omega_0\vert^2)\right)}}\omega_0\,,
$$
and   define $\omega_1$ in terms of $\omega_0.$
Likewise, formula \eqref{rho0} can be rewritten as
$$
\rho_1 (x, \vert\omega_1\vert^2) =e^{\zeta(\vert\omega_1\vert^2)-\eta (\vert\omega_1\vert^2)}\rho_0\left(x, f_\eta(\vert\omega_1\vert^2)\right)\,,$$
defining $\rho_1$ in terms of $\rho_0$.

It is easily seen that the differential form  $\omega_0$  satisfies system \eqref{NLHe} with density function $\rho_0 (x, \vert\omega_0\vert^2)$ if and only if $\omega_1$  satisfies system \eqref{HFg} with density function $\rho_1 (x, \vert\omega_1\vert^2)$ and coefficients $\eta$ and $\zeta$ as in the hypotheses of the theorem. In fact,
\[
e^\eta d\omega_0=e^\eta d\left(e^{-\eta}\omega_1\right)=- \,d\eta\wedge\omega_1+ d\omega_1\,,\; \mbox{ with }\eta = \eta(\vert\omega_1\vert^2) = \eta\left( g_\eta(\vert\omega_0\vert^2) \right)\,,
\]
and
\[
e^{\zeta} \,d\ast (\rho_0\omega_0) = e^{\zeta} \,d\ast (e^{-\zeta} \rho_1\omega_1)= -d\zeta\wedge\ast\left(\rho_1\omega_1\right)+d\ast(\rho_1\omega_1),
\mbox { with }\zeta = \zeta (\vert\omega_1\vert^2).\]
Finally, the ellipticity condition for system \eqref{NLHe} with $\rho = \rho_0$ is
$$\frac{\partial \phi_{\rho_0}(x, \hat t)}{\partial \hat t} >0\,,\qquad \mbox { with     }\; \phi_{\rho_0}(x,\hat t) =\, \hat t\,\rho_0^2(x, \hat t)\,,\quad \hat t\in f_\eta(I)\,.$$
Squaring both sides of \eqref{rho0} and multiplying by $\hat t$, one obtains
\begin{align}\notag
&\quad\phi_{\rho_0}(x,\hat t) =\,\hat t\, \rho_0^2 (x, \hat t) = \hat t \, e^{2\eta\left(g_\eta(\hat t)\right)}\left(\rho_1\left(x, g_\eta(\hat t)\right)e^{-\zeta\left(g_\eta(\hat t)\right)}\right)^2 = \notag \\
&\hat t \, e^{2\eta(t)} \left(   \rho_1\left(x, t\right)e^{-\zeta(t)}  \right)^2 = t \left(   \rho_1\left(x, t\right)e^{-\zeta(t)}  \right)^2 = \phi_{\rho_1 e^{-\zeta}} (x, t)\,;\,\mbox{ with }  t = g_\eta(\hat t\,).\notag\\
\notag\end{align}
Thus,
$$\frac{\partial \phi_{\rho_0}(x, \hat t\,)}{\partial \hat t} =
\frac{\partial \phi_{\rho_1 e^{-\zeta}}(x,  t)}{\partial t}\,\,  g_\eta^\prime(\hat t\,) ,\,\quad \mbox{ with } \, t = g_\eta(\hat t\,)\,. $$
\end{proof}
\begin{proposition} \label{use2} Let $\eta,\zeta\,:\, \Omega\to\R$ be prescribed continuously differentiable functions.  Then for every classical solution $\omega_1$ of  \eqref{HFg} with mass density $\rho_1$,  coefficients $\Sigma = d\zeta$ and  $\Gamma=d\eta$,  there is a classical solution  $\omega_0$ of the nonlinear Hodge equations \eqref{NLHe} with density $\rho_0.$ Here $\rho_0$ depends on
$\rho_1$, $\eta$ and $\zeta$; $\omega_0$  is related to $\omega_1$ by $C^1$ conformal transformations. The converse also holds. Ellipticity is preserved by this correspondence.
\end{proposition}
\begin{proof} Given a $k$-form $\omega_1$ and a density function $\rho_1(\vert\omega_1\vert^2)$, define
\[
\omega_0= e^{-\eta{(x)}}\omega_1\,\,\mbox{ and }\; \rho_0 (x,  \vert\omega_0\vert^2) = e^{\eta(x)-\zeta(x)}\rho_1(e^{2\eta(x)}\vert\omega_0\vert^2).
\]
If $\omega_1$ satisfies \eqref{HFg} with mass density $\rho_1$ and coefficients $\Sigma$ and  $\Gamma$ as in the hypotheses of the proposition, then $\omega_0$ satisfies \eqref{NLHe} with density $\rho_0$. In fact,
\[d\omega_0= d\left(e^{-\eta}\omega_1\right) =e^{-\eta} \left(- d\eta\wedge\omega_1+ d\omega_1\right)=0 \,,\]
\[d\ast \left(\rho_0\omega_0\right) = d(e^{-\zeta} \ast \rho_1\omega_1) = e^{-\zeta}\left(-d\zeta\wedge\ast\left(\rho_1\omega_1\right)+d\ast\left(\rho_1\omega_1\right)\right)=0
\,.\]
The converse, for prescribed $\omega_0$ and  $\rho_0$ holds with $\omega_1$ and $\rho_1$ defined by
$$
\omega_1= e^{\eta{(x)}}\omega_0\,\,\mbox{ and }\; \rho_1 (x,  \vert\omega_1\vert^2) = e^{\zeta(x)-\eta(x)}\rho_0(e^{-2\eta(x)}\vert\omega_1\vert^2)\,.
$$\end{proof}
The prescription  $\eta(x)=\zeta(x)$ in Proposition \ref{use2}, yielding the simpler relation between densities $\rho_0= \rho_1(\exp\left[2\eta(x)\right]\vert\omega_0\vert^2)$, corresponds to the variational equations for the  nonlinear Hodge--Frobenius theory for gradient-recursive $k$-forms    (that is, with prescribed exact $\Gamma$).

\section{Boundary value problems}
\label{S-BVP}
Theorem \ref{use1} -- with nonlinear Frobenius condition --  and Proposition \ref{use2} allow us to extend the existence and uniqueness theorem for the Dirichlet and Neumann problems established in \cite{SS2} for the conventional  nonlinear Hodge theory (proven in their strongest formulation for $1$-forms) to system  \eqref{HFg} for gradient-recursive forms. For this application we use the results in \cite{SS2} in their general formulation for
density functions which may depend explicitly on $x$.
Here  $M$ denotes an oriented, finite Riemannian manifold of dimension $n$ with $\mathcal C^\infty$ boundary \cite{ISS}.
The following theorems correspond to Theorems 1, 2  in
\cite{SS2}.
We establish the following definition.
\begin{definition} The triplet of functions $(\rho, \zeta, \eta)$ is said to be an \emph {admissible system} if the following conditions hold:
a) $f_\eta\,:\, t\in \R^+\cup\{0\}\to\hat t\equiv t \exp [-2\eta(t)] \in \R^+\cup\{0\}$ is  1:1 and onto;
b)  $\rho_0 \equiv \rho(x, t) \exp[\eta(t)- \zeta(x, t)]\in [k, 1/k]$ for some constant $k>0$, $\forall (x, t)$;
c) there exists  $T >0$ s.t.  $\left(\partial_t \phi_{\rho e^{-\zeta}}\right)\,  g^\prime_\eta  >0$,  $\forall x$, $\forall t\in (0, T)$; here $g_\eta$ denotes the inverse of $f_\eta$.
The {\emph sonic speed} associated with an admissible system $(\rho, \zeta, \eta)$ is $Q_s \equiv \sup \{T\}$ such that c) is satisfied.
 A $k$-form $\omega$ is said to be \emph{subsonic} if $\max_{x\in M} \vert\omega\vert^2 < Q_s.$
\end{definition}
Following \cite{SS2}, the inhomogeneous Dirichlet  boundary data are given by an element of the space
$\mathcal D = \ker d \oplus \mathcal C^{1+\alpha}(\bar M)$, while inhomogeneous Neumann data are given by an element of the space $\mathcal N = \ker d \oplus \mathcal N_2,$ with $\mathcal N_2 = \ker \delta$ if $n\leq 3$,
$\mathcal N_2 =0$ if $n>3$. We denote by $T\omega$, $N\omega$ respectively,  the restriction to the boundary of the tangential component, normal component respectively,
 of $\omega$.
\begin{theorem}
\label{DE!}
Let  $(\rho, \zeta, \eta)$ be an admissible system of class $\mathcal {C}^{2+\alpha}$ in $x$ and $\mathcal C^{1+\alpha}$ in $t$, with sonic speed $Q_s$.
There is an open connected set $\mathcal O\in \mathcal D$ containing the origin such that for each pair of $1$-forms $(\gamma, \sigma)\in \mathcal O$, there is a unique subsonic $1$-form $\omega\in \mathcal C^{1+\alpha}(\bar M)$ having the same relative periods as $\gamma$, satisfying  and
\begin{equation}\label{DHF}
\left\{\begin{array}{lll}
d* \left(\rho(\vert\omega\vert^2)\,\omega\right) =d\zeta \wedge * \left(\rho(Q)\,\omega\right) + d\ast \sigma\\
d\omega = d\eta\wedge \omega\\
T \left(e^{-\eta(\vert\omega\vert^2)} \omega\right) = T\gamma  \mbox { on } \partial M\,.
\end{array}\right.
\end{equation}
Moreover, for any given continuous path $(\gamma(\tau), \sigma(\tau))$ on $\mathcal D$, the solution $\omega(\tau)$ will also depend continuously on $\tau$ in the uniform norm and, either is subsonic $\forall \tau$ or there exists a number $\tau_s$ such that $\sup _{x\in M} \vert\omega\vert^2(\tau)\to Q_s$ as $\tau\to\tau_s$.
\end{theorem}
\begin{proof}
By Theorem \ref{use1}, system \eqref{DHF} is transformed into
\begin{equation}\label{DH}
\left\{\begin{array}{lll}
d* \left(\rho_0(x,\vert\omega_0\vert^2)\,\omega_0\right) = d\ast \sigma\\
d\omega_0 = 0\\
T \omega_0 = T\gamma  \mbox { on } \partial M\,,
\end{array}\right.
\end{equation}
with $\rho_0 (x, \hat t\,) = \exp[\eta(g(\hat t\,))-\zeta(x, g(\hat t\,))]\, \rho(x, g(\hat t\,))$. By Theorem \ref{use1}, $\rho_0$ is \emph{admissible} as defined in \cite{SS2}; that is, $\rho_0(x,t)\in [k, 1/k]$ and $\partial_{\hat t}\, \phi_{ \rho_0}>0$ $\forall \hat t\in (0, f_\eta (T))$. Therefore, the conclusions in Theorem 1 of \cite{SS2} extend to  \eqref{DHF}. By Theorem \ref{use1} the $1$-form $\omega = \exp[\eta\left(g(\vert\omega_0\vert^2)\right)]\,\omega_0$ is  the unique solution to \eqref{DHF} as required.
\end{proof}
 \begin{theorem}
\label{NE!}
Let  $(\rho, \zeta, \eta)$ be an admissible system of class $\mathcal {C}^{2+\alpha}$ in $x$ and $\mathcal C^{1+\alpha}$ in $t$, with sonic speed $Q_s$. There is an open connected set $\mathcal O\in \mathcal N$ containing the origin such that for each pair of $1$-forms $(\gamma,  \nu)\in \mathcal O$, there is a unique subsonic $1$-form $\omega\in \mathcal C^{1+\alpha}(\bar M)$ having the same absolute periods as $\gamma$,  satisfying
\begin{equation}\label{NHF}
\left\{\begin{array}{ll}
d* \left(\rho(x, \vert\omega\vert^2)\,\omega\right) =d\zeta \wedge * \left(\rho(Q)\omega\right) \\
d\omega = d\eta\wedge \omega\\
N (\rho \,e^{-\zeta} \omega) = N \nu \mbox { on } \partial M\,.
\end{array}\right.
\end{equation}
Moreover, for any given continuous path $(\gamma(\tau), \nu (\tau))$ on $\mathcal O$, the same conclusions as in Theorem \ref {DE!} hold for the path of solutions $\omega(\tau)$.
\end{theorem}
\begin{proof}
By Theorem \ref{use1}, system \eqref{NHF} is transformed into
\begin{equation}\label{DH}
\left\{\begin{array}{lll}
d* \left(\rho_0(x,\vert\omega_0\vert^2)\,\omega_0\right) = 0\\
d\omega_0 = 0\\
N (\rho_0\omega_0) = N\nu  \mbox { on } \partial M\,,
\end{array}\right.
\end{equation}
with $\rho_0 (x, \hat t\,) = \exp[\eta(g(\hat t\,))-\zeta(x, g(\hat t\,))] \rho(x, g(\hat t\,))$, satisfying the hypotheses of Theorem 2 of \cite{SS2}.  Again, by Theorem \ref{use1} $\omega = \exp[\eta\left(g(\vert\omega_0\vert^2)\right)]\,\omega_0$ is  the unique solution to \eqref{NHF} as required.
\end{proof}
\begin{remark}
For a linear  Frobenius condition, that is,  if $\eta= \eta(x)$, simpler versions of Theorems \ref{DE!} and \ref{NE!} hold and their proofs are a direct application of Proposition \ref{use2} to Theorems 1, 2 in \cite{SS2}. For simplicity, we have not  addressed the question on whether the surjectivity hypothesis on $f_\eta$ can be removed in Theorems \ref{DE!} and \ref{NE!}.
\end{remark}
\begin{remark} It is natural to expect that, at least in the case of the variational equation \eqref{NHEn},  Theorem \ref{use1} and Proposition \ref{use2} would lead to decomposition theorems for gradient-recursive
differential forms, mirroring the conventional nonlinear Hodge decomposition theorems.
Furthermore, the duality result of Proposition \ref{HBg} has potential importance in extending  nonlinear Hodge decomposition theorems  to include differential forms satisfying the nonlinear Hodge--Frobenius equations that are
not necessarily gradient-recursive.  Because all recursive forms of degree or co-degree $1$ are gradient-recursive, this investigation would be of special interest for applications to forms of degree $k\neq 1, n-1$.   In this  regard, we observe that the Frobenius theorem for 1-forms,  stating that 1-forms that generate a closed ideal are integrable, does not extend to $k$-forms with $k\neq 1, n-1$. \end{remark}
\textbf{Acknowledgment}. The authors are grateful to a referee for constructive criticism of an earlier draft of this paper.

\end{document}